\documentclass[reqno,11pt]{amsart}
\usepackage{amsmath,amssymb,amsthm, amscd, braket}
\usepackage[labelformat=empty]{caption}
\usepackage{array}
\usepackage{soul} 
\usepackage{color, xcolor}

\usepackage{amsbsy}
\usepackage{latexsym}
\usepackage[all]{xy}
\usepackage{amsthm}
\usepackage{mathrsfs}
\usepackage{dsfont}

\usepackage{color}

\usepackage{hyperref}
\usepackage{color}
\hypersetup{
    colorlinks=true,
    citecolor=blue,
    linkcolor=black,
  urlcolor=blue,
}

\usepackage{amsbsy}
\usepackage{latexsym}
\usepackage[all]{xy}
\usepackage{amsthm}
\usepackage{mathrsfs}
\usepackage{dsfont}

\usepackage{color}

\usepackage{color, xcolor}
\usepackage{hyperref}
\usepackage{color}
\hypersetup{
    colorlinks=true,
    citecolor=blue,
    linkcolor=black,
  urlcolor=blue,
}

\usepackage{ulem}

\usepackage{amscd,amsthm,amsfonts,amssymb,amsmath}
\usepackage{amsmath,tikz-cd}
\usepackage{fullpage}
\usepackage[all]{xy}
\usepackage{mathtools}
\usepackage{mathrsfs}
\usepackage{nccmath}
\usepackage{amscd} 
 
\newcommand{\C}{\BC}

    \newcommand{\BC}{{\mathbb {C}}}

    \newcommand{\BQ}{{\mathbb {Q}}}

     \newcommand{\BZ}{{\mathbb {Z}}}

    \newcommand{\CA}{{\mathcal {A}}}

     \newcommand{\CL}{{\mathcal {L}}}
     
    \newcommand{\CO}{{\mathcal {O}}}

     \newcommand{\CV}{{\mathcal {V}}}

    \newcommand{\Aut}{{\mathrm{Aut}}}

    \newcommand{\End}{{\mathrm{End}}}

    \newcommand{\Gal}{{\mathrm{Gal}}} \newcommand{\GL}{{\mathrm{GL}}}
    
    \newcommand{\Hom}{{\mathrm{Hom}}}

    \newcommand{\rank}{{\mathrm{rank}}}

    \newcommand{\Q}{\mathbb{Q}}

    \font\cyr=wncyr10

    \newcommand{\Sha}{\hbox{\cyr X}}

    \newcommand{\ov}{\overline}

    \newcommand{\ra}{\rightarrow}

    \theoremstyle{plain}
    \newtheorem{thm}{Theorem}[section] \newtheorem{cor}[thm]{Corollary}
      \newtheorem{prop}[thm]{Proposition}
    \newtheorem {conj}[thm]{Conjecture} \newtheorem{defn}[thm]{Definition}

\theoremstyle{remark} \newtheorem{remark}[thm]{Remark}
\theoremstyle{remark} 
\theoremstyle{remark}

    \newcommand{\cO}{\mathcal O}
    \numberwithin{equation}{section}
  \date{\today}

\begin{document} 
\title{A refined non-vanishing of the $p$-adic logarithm of a  rational point on an abelian variety}
\dedicatory{À Henri Darmon, avec admiration. \\ To Henri Darmon, with admiration. \\ }
\author{Ashay Burungale}
\address{Department of mathematics, University of Texas at Austin, 
2515 Speedway, Austin TX 78712} 
\email{ashayburungale@gmail.com}
\author{Christopher Skinner}
\address{Department of Mathematics, Princeton University, 
Princeton NJ 08544-1000}
\email{cmcls@princeton.edu}

\author{Xin Wan}
\address{Morningside Center of Mathematics; Academy of Mathematics and Systems Science, Chinese
Academy of Sciences, Beijing 100190, China}
\email{xwan@math.ac.cn}
\date{}

\begin{abstract}
Inspired by a beautiful formula of Bertolini, Darmon, and Prasanna -- the oft-termed {\it BDP formula} -- we address questions about the non-vanishing of non-torsion points under $p$-adic logarithms of abelian varieties. We largely consider situations most applicable to $\GL_2$-type abelian varieties associated with Hilbert modular newforms and Heegner points. Not surprisingly, the main tool employed is the $p$-adic analytic subgroup theorem.

    %Let $A$ be an abelian variety defined over a number field. Suppose that $F\hookrightarrow \End^{0}(A)$ for a number field $F$ such that $[F:\BQ]=\dim A$. For a non-torsion $x\in A(\ov{\BQ})$ and a prime $p$, we prove that 
    %$$
    %\log_{A,\sigma}(x) \neq 0
    %$$ 
    %for any embedding $\sigma: F \hookrightarrow \ov{\BQ}$ under some conditions, where $\log_{A,\sigma}$ is the $\sigma$-component of the $p$-adic logarithm $\log_{A}$. We also consider the $p$-adic closure of rational points on $A$ and determine its rank under the structural rank conjecture. 
%    This non-vanishing is an ingredient in the proof of the $p$-part of the conjectural Birch and Swinnerton-Dyer formula for certain $\GL_2$-type abelian varieties over $\BQ$. 
\end{abstract}

\newtheorem{theorem}{Theorem}
\newtheorem{observation}{Observation}

\maketitle
\tableofcontents
\section{Introduction}

\subsection{The BDP formula}
In \cite{BDP1} Darmon, in joint work with Bertolini and Prasanna, proved a beautiful formula that expresses a special value of a certain $p$-adic $L$-function of a modular newform 
as (the square of) the image under a $p$-adic Abel--Jacobi map of a generalized Heegner cycle. 

In the special case that the newform $f$ has weight $2$, level $N$, and trivial Nebentypus, this formula takes the shape:
\begin{equation}\label{eq:BDP1}
L_p(f,\mathbf{N}_K) = (1-a_p(f)+p^{-1})^2 (\log_{\omega_{B_f}}(P_f))^2
\end{equation}  (cf.~\cite[Thm.~ 3.12]{BDP2}).
This has come to be known as the {\it BDP formula}.
Here, $p\nmid 2N$ is a prime, $K/\Q$ is an imaginary quadratic field in which all primes dividing $Np$ split (so $K$ satisfies a `Heegner hypothesis' relative to $f$), 
and $L_p(f,-)$ is a continuous function of a $p$-adic space of ($p$-adic) Hecke characters over $K$ that takes values in $\BC_p$.
The function $L_p(f,-)$ interpolates the algebraic parts of special $L$-values $L(f,\chi^{-1},0)$ for $\chi$ in a particular set of Hecke characters over $K$. The norm character $\mathbf{N}_K$ belongs to the $p$-adic space of characters but does not belong to the set with interpolated special $L$-values. 
On the other side of the formula, $B_f$ is an abelian variety over $\Q$ in the isogeny class associated with $f$ that is realized as a quotient $\Phi_f:J_1(N)\twoheadrightarrow B_f$
and $P_f\in B_f(K)$ is the Heegner point. The differential $\omega_{B_f}$ is the unique differential in $\Omega^1_{B_f}\otimes_{\BQ}\ov{\BQ}$ that pulls back under $\Phi_f$ to the differential $\omega_f = 2\pi i f(\tau)d\tau$ in $\Omega^1_{J_1(N)}\otimes_{\BQ}\ov{\BQ}= \Omega^1_{X_1(N)}\otimes_{\BQ}\ov{\BQ}$.   The logarithm $\log_{\omega_{B_f}}$ is the $p$-adic logarithm associated with the differential. Hidden in all this is the choice of an embedding $\iota_p: \ov{\BQ}\hookrightarrow \ov{\BQ}_p$ via which the algebraic parts
of the special $L$-values are viewed as belonging to $\C_p$ and via which the differential $\omega_f$ is viewed as belonging to $\Omega^1_{B_f}\otimes_{\BQ}\ov{\BQ}_p$; there is a $p$-adic $L$-function and a formula for each choice of $\iota_p$.  

The BDP formula \eqref{eq:BDP1} was inspired by an earlier, similar formula of Rubin for CM elliptic curves \cite{Ru}, and in \cite{BDP2} Darmon and his collaborators gave a new proof Rubin's formula based on \eqref{eq:BDP1}.  As was beautifully explained in \cite{BCDDPR}, this formula can also be seen as a special case of Perrin-Riou's $p$-adic Beilinson conjectures and is closely analogous to $p$-adic regulator formulas for the $p$-adic $L$-functions of Kubota--Leopoldt and Katz, which involve $p$-adic logarithms of circular and elliptic units, respectively.  The BDP formula has also played a central role in progress toward the Birch--Swinnerton-Dyer conjecture for elliptic curves, especially in conjunction with the Gross--Zagier formula, the methods of Kolyvagin, and Iwasawa theory (see in particular \cite{S,BS1,W,JSW,BST,BSTW,BT,BCST,BCGS}).  The question addressed in this note arises naturally from the BDP formula and a wish to extend the results for elliptic curves to the abelian varieties associated with a weight $2$ newform. 

\subsection{The question}
Suppose the Heegner point $P_f \in B_f(K)$ is non-torsion, which happens if and only if the $L$-function $L(f/K,s)$ has a zero of order one at $s=1$ by the Gross--Zagier formula \cite{GZ,YZZ}.
Then it is natural to ask: 
$$
\text{Is its $p$-adic logarithm $\log_{\omega_f}(P_f)$ non-zero?} 
$$

Indeed, there are good reasons coming from Iwasawa theory to expect that this is so.  For example, the $p$-adic height of $P_f$ appearing in Kobayashi's $p$-adic Gross--Zagier formula \cite{Kob,BKO} (for the case that $f$ is {non-ordinary} at $p$) can only be non-zero if $\log_{\omega_f}$ is.  Other reasons arise from considering the Iwasawa-theoretic main conjecture
associated with $p$-adic $L$-function $L_p(f,-)$ appearing in the BDP formula \eqref{eq:BDP1}, in which case the answer `yes' can be connected to the truth of the expected
BSD formula for the derivative at $s=1$ of the complex $L$-function $L(f,s)$ 
(see \cite{S,JSW}). 
Moreover, in the case that $B_f$ is an elliptic curve $E$, the answer is clearly yes:
$E(K)$ injects into $E(\ov{\BQ}_p)$ and the kernel of the $p$-adic logarithm $\log_{\omega_E}:E(\ov{\BQ}_p)\rightarrow \ov{\BQ}_p$ is just the torsion points. 

The subtleties appear when $B_f$ is not an elliptic curve, as we will explain below. In \cite[Lem.~2.2.2]{S}, one of us (C.S.) implicity asserted a positive answer to this question (this has no impact on the main results of \cite{S}). Another of us (X.W.) later pointed out that no proof was given. In this note we supply a proof that the answer to this question is `yes.' Perhaps not surprisingly, our proof is based on results from $p$-adic transcendence theory -- the $p$-adic analytic subgroup theorem, in particular.

\subsection{The case of general $B_f$} To simplify notation, we will now write $A$ for $B_f$. Let $\iota_\infty:\ov{\BQ}\hookrightarrow\C$ be a fixed embedding. Then via $\iota_\infty$, the Fourier coefficients of $f$ generate a totally real number field $F\subset \ov{\BQ}$, the dimension of $A$ equals the degree $[F:\BQ]$ of $F$, and the endomorphism ring $\End_\Q(A)$ contains an order $\CO$ of $F$.  
Let $t_A$ be the tangent space of $A$ and recall that the space of differentials $\Omega^1_A$ is canonically identified with the dual of $t_A$. In particular,
$\Omega = \Omega^1_A\otimes_{\BQ}\ov{\BQ}$ is the $\ov{\BQ}$-dual of $V = t_A\otimes_{\BQ}\ov{\BQ}$; given $\omega\in \Omega$ we write $\ell_\omega: V\rightarrow\ov{\BQ}$
for the corresponding linear map.  The action of $\CO$ on $A$ induces an action of $F$ on $\Omega^1_A$ and hence an action of $F\otimes_{\BQ}\ov{\BQ}$ on $\Omega$. 
Let $\Sigma_F = \{\sigma:F\hookrightarrow \ov{\BQ}\}$ be the set of embeddings of $F$ into $\ov{\BQ}$.  The decomposition 
$F\otimes_{\BQ}\ov{\BQ} = \oplus_{\sigma\in \Sigma_F}\ov{\BQ}e_\sigma$, where $e_\sigma^2=e_\sigma$ and the action of $F$ on $e_\sigma$ is via the embedding $\sigma$,
induces a decomposition $$\Omega = \oplus_{\sigma \in\Sigma_F} \Omega_\sigma,$$ with $\Omega_\sigma = e_\sigma \Omega$.  For each $\sigma\in \Sigma_F$, let 
$f^\sigma$ be the newform whose Fourier coefficients are obtained from $f$ by applying $\sigma$. Then $\omega_{f^\sigma}$ is a $\ov{\BQ}$-basis for $\Omega_\sigma$.

Let $p$ be a prime and let $\iota_p:\ov{\BQ}\hookrightarrow \ov{\BQ}_p$ be a fixed embedding.
Let $V_p = t_A\otimes_{\ov{\BQ}}\ov{\BQ}_p$. The $p$-adic logarithm of $A$ is a locally analytic homomorphism $$\log_A: A(\ov{\BQ}_p)\rightarrow V_p$$ whose
kernel is the group of torsion points. The logarithm is $\CO$-invariant.  The embedding $\iota_p$ induces an inclusion $A(\ov{\BQ})\hookrightarrow A(\ov{\BQ}_p)$ as well as an identification $V_p = V\otimes_{\ov{\BQ}}\ov{\BQ}_p$.  In particular, each $\ell_\omega$, $\omega\in \Omega$ extends to
a $\ov{\BQ}_p$-linear map $\ell_\omega:V_p\rightarrow \ov{\BQ}_p$.  We set $\log_\omega:A(\ov{\BQ}_p)\rightarrow \ov{\BQ}_p$ to be the composition 
$$\log_\omega = \ell_\omega\circ\log_A.$$  If $\omega = \omega_{f^\sigma}$, then this is the $p$-adic logarithm appearing in the BDP formula.
Note that the definition of $\log_\omega$ depends on the choice of $\iota_p$. 
Replacing $\iota_p$ with some $\iota_p'$ has the effect of replacing $\log_{\omega_{f^\sigma}}$ 
(with respect to $\iota_p$) with $\log_{\omega_{f^{\sigma_1}}}$ for some $\sigma_1$ (but again with respect to $\iota_p$), and 
all $\sigma_1\in\Sigma_F$ can be reached this way. This is not an issue if $F=\Q$, but is otherwise.  As there is no
distinguished $\iota_p$, to properly answer the question one should consider all the $\log_{\omega_{f^\sigma}}$'s. 
This should not be surprising: this is just answering the question for all newforms in the Galois orbit of $f$ and these all have the same
associated abelian varieties. 

Let $x\in A(\ov\BQ)$ be a non-torsion point. Then it is clear that $\log_A(x)$ is non-zero and hence that $\log_{\omega_{f^\sigma}}(x)\neq 0$ for {\it some} $\sigma$ but possibly not all. In this note we address the question: 
$$
\text{Is $\log_{\omega_{f^\sigma}}(x) \neq 0$ for all $\sigma\in\Sigma_F$?}
$$
We give a positive answer to this. In fact, we address a more general question that also applies to abelian varieties associated with Hilbert modular forms and with Heegner points that come from twisting by finite order ring class characters over CM fields.

\subsection{A representative result}
The fact that $F$ was totally real played no obvious role in the above discussion, nor did the fact that $A$ was defined over $\BQ$. We can make the same constructions and ask the same questions for any abelian variety $A/\ov{\BQ}$ with an embedding $F\hookrightarrow \End^0_{\ov{\BQ}}(A)$ of a field $F$. In this context we prove:

\begin{thm}\label{thm:A} Let $A/\ov{\BQ}$ be an abelian variety. Suppose there is a field $F\subset \End^0_{\ov{\BQ}}(A)$ such that 
\begin{itemize}
\item[(a)] $\dim(A) = [F:\Q]$,
\item[(b)] $F$ has at least one real embedding.
\end{itemize}
For any non-torsion $x\in A(\ov{\BQ})$, we have 
$$\log_{\omega}(x)\neq 0$$ for all $0\neq \omega\in \Omega_\sigma$ and all $\sigma\in \Sigma_F$. 
\end{thm}
In particular, this theorem gives a positive answer to the two questions displayed above.
We actually prove more general theorems. These are stated in 
Section \ref{s, mr}; see especially Theorem \ref{thm,mr'} and Remark \ref{rmk:impliesA}. The proofs are given in Section \ref{s, proofs}.
We also discuss a motivating example arising from 
CM modular forms (see Section \ref{ss,ex}) and an arithmetic application, to the BSD formula for the modular
abelian variety $B_f$ in the case that $L(f,s)$ has analytic rank one (see Section \ref{ss, p-BSD}). 

\subsection{Some additional context} 
Let $G/\BQ$ be a commutative algebraic group.
In \cite{Po} Poonen asked an intriguing question about the dimension (as a $p$-adic Lie group) of the $p$-adic closure $\ov{\Gamma}$ in $G(\BQ_p)$ of a finitely-generated subgroup
$\Gamma\subset G(\BQ)$ contained in the union of all compact subgroups of $G(\BQ_p)$. This question is also independently due to Prasad \cite{Pr}.
If $G=A$ is an abelian variety that is $\BQ$-simple and $\Gamma = A(\BQ)$, then a special case of this question asks:
\begin{equation*}
\text{Is $\dim(\ov{\Gamma}) = \min\{\dim(A),\rank A(\BQ)\}$?}
\end{equation*}

The best result to date towards answering this is due to Waldschmidt \cite{Wa}, who proved that if $A$ is $\ov{\BQ}$-simple,
then $\dim\ov{A(\Q)} \geq (\rank A(\BQ) \cdot \dim(A))/(\rank A(\BQ) + 2\dim(A))$ (so in particular, at least $\frac{1}{3}\min\{\rank A(\Q), \dim(A)\}$).
However, if $A$ is as in Theorem \ref{thm:A} with the action of $F$ defined over $\Q$, then the conclusion of Theorem \ref{thm:A} shows that this question has a positive answer in this case.  

In the spirit of Poonen's or Prasad's question, in the final section of this note, we explain how a simple variant of the proof of Theorem \ref{thm:A} proves:
\begin{thm}\label{thm:B}
Let $A/\BQ$ be an abelian variety and suppose $F = \End^0_{\ov\BQ}(A)$ is a field (so $F$ is totally real or CM). 
Let $x_1,...,x_r\in A(\ov{\BQ})\otimes_{\BZ}\Q$ be $F$-linearly independent
and set $\Gamma = F\cdot x_1+\cdots+ F\cdot x_r$. Assuming the structural rank Conjecture~\ref{conj,str}, the dimension of the $\ov{\BQ}_p$-space $W_p\subset V_p$
spanned by $\log_A(\Gamma)$ equals $\min\{r\deg(F),\dim(A)\}$ if $F$ is totally real and is at least $\min\{r\deg(F)/2,\dim(A)\}$ if $F$ is CM. 
\end{thm}
From the perspective of this theorem, the answer to the first two displayed questions boil down to the structural rank conjecture being trivially true for $1\times 1$-matrices!
As in the proof of Theorem \ref{thm:A}, the heavy lifting is done by the $p$-adic analytic subgroup theorem.

\subsection*{Acknowledgments} We than Dipendra Prasad and the referee for helpful suggestions on the note. 

The works of Henri Darmon and his unerring sense of fruitful questions and enthusiasm for mathematical exploration
have been and continue to be an inspiration to us. It is our pleasure to dedicate this note to Henri and express our sincere gratitude to him and especially to include it in this volume marking his 60th birthday.

During the preparation of this paper, A.B. was partially supported by the NSF grant DMS-2302064
%; C.S. was partially supported by the Simons Investigator Grant \#376203 from the Simons Foundation and by the NSF grant DMS-1901985 
and X.W. was partially supported 
by NSFC grants 12288201, 11621061, CAS Project for Young Scientists in Basic Research grant no. YSBR-033,  
National Key R\&D Program of China 2020YFA0712600 and the Strategic Priority Research Program of Chinese Academy of Sciences under Grant XDA0480503.

\section{Main results: statements} \label{s, mr}
In this section we state our main results. The proofs appear in section \ref{s, proofs}. 
Let $\iota_\infty:\ov{\BQ}\hookrightarrow \C$ and $\iota_p:\ov{\BQ}\hookrightarrow \ov{\BQ}_p$ be fixed embeddings.
We use the former to identify $\ov{\BQ}$ with a subfield of $\C$.

\subsection{The set-up}\label{ss,st}
Let $A$ be an abelian variety over a number field $L$.
Suppose that there exists a number field $E$ and an embedding  
$$
\theta: E\hookrightarrow \mathrm{End}^0_L(A).
$$
%We assume that
%\begin{equation*}
%\mathrm{(a)} \ \ \ \mathrm{dim}(A) = [E:\mathbb{Q}].
%\end{equation*}

Let $t_A$ be the tangent space of $A$. Then $t_A$ is an $E\otimes_{\BQ} L$-module, with the action of 
$E$ induced by $\theta$ and the action of $L$ the usual scalar action. Let $\Omega^1_A$ be the differentials
of $A$. This, too, is an $E\otimes_{\BQ} L$-module. Furthermore, $\Omega_A^1$ is canonically identified with the $L$-dual 
of $t_A$:  $\Omega_A^1 = \Hom_L(t_A,L)$, and the $E\otimes_{\BQ} L$-action on $\Omega_A^1$ is just that induced
from the action on $t_A$.

Let $\Sigma_E = \{\sigma:E\hookrightarrow \ov{\BQ}\}$ be the set of embeddings of $E$.
Recall that there is a canonical decomposition $E\otimes_{\BQ}\ov{\BQ} = \oplus_{\sigma\in\Sigma_E} \ov{\BQ}e_\sigma$, 
where $e_\sigma^2 = e_\sigma$ and $E$ acts on $e_\sigma$ via the embedding $\sigma$.

%In some situations we might also assume that  
%\begin{itemize}
%\item[(b)] $t_A$ is a free $E\otimes L$-module of rank $1$ (equivalently, $\Omega_A^1$ is a free $E\otimes L$-module of rank $1$).
%\end{itemize}

Let $H/L$ be a finite extension of fields. Fix an embedding $\tau: H\hookrightarrow \ov{\mathbb{Q}}$.
Let $V = t_A\otimes_{L,\tau}\ov{\BQ}$ and let $\Omega = \Omega_A^1\otimes_{L,\tau}\ov{\BQ}$. 
Let $V_{\sigma} = e_\sigma V$ and $\Omega_{\sigma} = e_\sigma\Omega$, so $V= \oplus_\sigma V_{\sigma}$ and $\Omega = \oplus_\sigma \Omega_{\sigma}$.
Under $\Omega = \Hom_{\ov{\BQ}}(V,\ov{\BQ}) = \oplus_\sigma\Hom_{\ov{\BQ}}(V_{\sigma},\ov{\BQ})$, 
$\Omega_{\sigma}$ is identified with $\Hom_{\ov{\BQ}}(V_{\sigma},\ov{\BQ})$.
Given $\omega\in \Omega$ we write $\ell_\omega: V\rightarrow \ov{\BQ}$ for the corresponding $\ov{\BQ}$-linear map.

Let $V_p = V\otimes_{\ov{\BQ},\iota_p}\ov{\BQ}_p$. Note that
$V_p = \oplus_{\sigma} V_{p,\sigma}$ for $V_{p,\sigma} = e_\sigma V_p = V_\sigma \otimes_{\ov{\BQ},\iota_p}\ov{\BQ}_p$.
Also, each $\ell_\omega$ extends to a $\ov{\BQ}_p$-linear map $\ell_\omega: V_p\rightarrow \ov{\BQ}_p$; if $\omega \in \Omega_\sigma$, 
then $\ell_\omega$ factors through the projection to $V_{p,\sigma}$.

Let $\log_A: A(\ov{\BQ}_p)\rightarrow V_p$
be the $p$-adic logarithm for $A$. For $\omega\in \Omega$ we let
$$
\log_\omega = \ell_\omega\circ\log_A: A(\ov{\BQ}_p)\rightarrow \ov{\BQ}_p.
$$
This is the $p$-adic logarithm associated with $\omega$.  Our aim is understand when $\log_\omega(x)\neq 0$ for $x\in A(H)$ non-torsion and $0\neq \omega\in \Omega_\sigma$
for some $\sigma$. Before we state our main result in this direction we need a couple of preliminary observations.

\subsubsection{Trace fields}\label{ss, tf}
Suppose that $\dim(A) = [E:\BQ]$. Let $B\subset A$ be a proper non-zero $E$-stable abelian subvariety (not necessarily defined over $L$). 
Here by `$E$-stable' we mean in the category where we have 
replaced $\Hom(X,Y)$ with $\Hom^0(X,Y) = \Hom(X,Y)\otimes_{\BZ}\BQ$ for two abelian varieties $X$, $Y$. 
Then $\dim(B) < [E:\Q]$, so it follows that\footnote{by considering $H_1(B(\C),\Q)$ as an $E$-space} $2\dim(B) = [E:\Q]$. 
It also follows that $E$ has no real embeddings\footnote{by considering $H^1(B(\C),\Q) = t_{B/\C}\oplus \ov{t}_{B/\C}$}.
Let $B' = A/B$. This is also an abelian variety with an induced $E$-action. It follows that there is an isogeny 
$$A\sim B\times B'$$ that is $E$-equivariant for the diagonal action of $E$ on the product. So the $E$-action on 
$A$ extends to an $E\times E$-action on $A$, with the original action coming from the diagonal embedding. Note that
the extended action may not be defined over $L$.
Moreover, if such an extension to an $E\times E$-action exists, then an $E$-stable abelian subvariety $B$ clearly exists such that 
$E\times E$ also acts on $B$ through projection to the first factor.
\begin{defn}
We define the trace field of an $E\times E$-extension as above to be the extension of $L$ generated by the values
of the traces of the elements of $E$ acting on the tangent space $t_{B/\ov{\BQ}}$ of $B/\ov{\BQ}$. 
\end{defn}
Since $V = t_{A/\ov{\BQ}} = t_{B/\ov{\BQ}}\oplus t_{B'/\ov{\BQ}}$ and the 
$E$-action on $A$ is defined over $L$, the trace field is the same as the extension generated by the traces of the elements of $E$ acting on the tangent space $t_{B'/\ov{\BQ}}$.  It is also just the extension of $L$ generated by the traces of the elements of $E\times 0 \subset E\times E$ acting
on $V$. 
If $B''$ is any other $E$-stable
subvariety, then one of the projections $B''\rightarrow B$ or $B''\rightarrow B'$ induced from the isogeny
$A\sim B\times B'$ must be non-zero and hence an $E$-equivariant isogeny. Therefore, the 
trace field of an $E\times E$-extension does not depend on the particular extension, just that such an extension exists.

\begin{remark}\label{rmk, tf} 
If $E$ is a CM field then the trace field is the composition of $L$ with the reflex field of the action of the CM-algebra $E\times E$. This is immediate from the definitions.
\end{remark}

%Note that $V$ is a free $E\otimes\ov{\mathbb{Q}}_p$-module of rank $1$. We {\color{red}decompose}
%$V=\oplus_{\sigma: E\hookrightarrow \ov{\mathbb{Q}}_p} V_\sigma,$
%where 
%$V_\sigma=t_A\otimes_{E,\sigma}\ov{\mathbb{Q}}_p$
%is a $1$-dimensional $\ov{\mathbb{Q}}_p$-space.
%Let
%$$\mathrm{log}_{A,\sigma}: A(\ov{\mathbb{Q}}_p)\rightarrow V_\sigma$$
%be the projection of $\mathrm{log}_{A}$ to $V_\sigma$. Given a non-torsion point $x\in A(H)$, a basic problem is whether $\mathrm{log}_{A,\sigma}(x)\not=0$ for any $\sigma: E \hookrightarrow \ov{\BQ}_p$ (cf.~\eqref{Q}).
%
%

\subsection{Main results}
We can now state our first main result. 

\begin{thm}\label{thm,mr'} Suppose that
\begin{itemize}
\item[(a)] $\dim(A) = [E:\BQ]$, 
\item[(b)] if the embedding $\theta: E\hookrightarrow \mathrm{End}^0_L(A)$ extends to an embedding $E\times E \hookrightarrow \End^0_{L}(A)$,
with $\theta$ being the restriction to the diagonal, then $H$ does not contain the trace field of this extension.
\end{itemize}
Then for all non-torsion $x\in A(H)$, all $\sigma \in \Sigma_E$, and all $0\neq \omega \in \Omega_\sigma$ 
we have 
$$
\log_{\omega}(x)\neq 0.
$$
\end{thm}

\begin{remark}\label{rmk:impliesA}
Note that this theorem implies Theorem \ref{thm:A} by taking $E = F$, $L$ any field over which $A$ is defined and $H$ any extension of $L$ over which the given non-torsion point $x$ is defined. The key point is that the hypothesis in Theorem \ref{thm:A} that $F$ has a real embedding implies that $\theta$ has no extensions to $F\times F$ (see \S\ref{ss, tf}), so hypothesis (b) here is automatic.
\end{remark}

Our next main theorem is:

\begin{thm}\label{thm,mr2'} Suppose that 
\begin{itemize}
\item[(a)] $\dim(A) = [E:\BQ]$,
\item[(c)] $A$ admits complex multiplication 
by a quadratic extension $E'/E$ (not necessarily over $L$) with $E'$ a CM field.
\end{itemize}
Then for all non-torsion $x\in A(H)$, all $\sigma \in \Sigma_E$, and all $0\neq \omega \in \Omega_\sigma$ 
we have 
$$
\log_{\omega}(x)\neq 0.
$$  
\end{thm}

We can also partly deal with cases where hypothesis (b) of Theorem \ref{thm,mr'} does not hold.

\begin{thm}\label{thm,mr3'}
Suppose
\begin{itemize}
\item[(a)] $\dim(A)=[E:\Q]$.
\end{itemize}
Let $x\in A(H)$ be a non-torsion point. Then there exists a smallest $E$-stable abelian subvariety $B\subset A$ (not necessarily defined over $L$) containing some multiple $mx$ of $x$ for some integer $m\neq 0$.
Furthermore, either (i) $B=A$, or (ii) $2\dim(B)=\dim(A)$. Also, 
for all $\sigma$ and all $\omega\in \Omega_\sigma$ such that $\ell_\omega|t_{B}\neq 0$, 
$$
\log_\omega(x)\neq 0.
$$
\end{thm}
The conclusion of this theorem can be rephrased as: $\log_{\omega'}(x)\neq 0$ for all $\sigma$ and all $0\neq\omega' \in \Omega'_\sigma$, where $\Omega' = \Omega_{B/\ov{\BQ}}^1 = \oplus_\sigma \Omega'_\sigma$ and $\log_{\omega'}$ is the associated $p$-adic logarithm of $B$.

\begin{remark} If $E$ is a CM field, then the $\sigma$'s such that 
$\Omega'_\sigma\neq 0$ comprise a CM-type $\Phi\subset\Sigma_E$.
Of course, it could be that there are $\sigma$'s in the complement $\ov\Phi$ such that $\Omega_\sigma\neq 0$ (for example if $A = B\times B'$ with $B'$ having CM by $E$ with CM-type $\ov{\Phi}$).
For such $\sigma$ and $\omega\in \Omega_\sigma$ we will always have $\log_\omega(x) = 0$.
\end{remark}

\subsection{Applications to $\GL_2$-type abelian varieties}
We have already noted that Theorem \ref{thm,mr'} implies Theorem \ref{thm:A} and hence that the answer to the first two displayed questions in the introduction is `yes.' We now deduce further consquences related to the analogous questions for twists of Heegner points by finite order Hecke characters. 

Let $A_0/L$ be an abelian variety over a totally real field $L$ and suppose that there is an embedding $\theta_0:F\hookrightarrow \End^0_L(A_0)$
of a totally real field $F$ such that $\dim(A_0)=[F:\Q]$.
Let $\CO = F\cap \End_L(A_0)$; this is an order in $F$. Replacing $A_0$
with $\Hom_{\CO}(\CO_F,A_0)$ if necessary, we can assume that $\CO$ is the maximal order $\CO_F$ (the ring of integers of $F$). This amounts to replacing $A_0$ with another abelian variety in its $L$-isogeny class. Let $E/F$ be either a totally real or CM field containing $F$ (not necessarily a quadratic extension). Let 
$$
A=A_0\otimes_{\CO_F}\CO_E.
$$
This is an abelian variety over $L$.
Then $\theta_0$ extends to an embedding $\theta:E\hookrightarrow \Hom^0_L(A)$: the order $\CO_E$ acts in the obvious way on
$A_0\otimes_{\CO_F}\CO_E$, that is, by multiplication on the second factor.

Let $K/L$ be a finite extension and $H/K$ a finite abelian extension. 
Let $\chi:\Gal(H/K)\rightarrow \CO_E^\times$ be a character (which is either trivial or quadratic if $E$ is totally real, but otherwise has no restriction on its order). Let
$y\in A_0(H)$ be a non-torsion point. We then consider
$$
x = \sum_{\tau\in \Gal(H/K)} \tau(y)\otimes\chi(\tau)\in A_0(H)\otimes_{\CO_F}\CO_E = A(H).
$$
Suppose $x$ is non-torsion. 

\begin{thm}\label{thm,nv}
Suppose one of the following holds:
\begin{itemize}
    \item[(i)] $E$ is totally real, 
    \item[(ii)] $A_0$ does not have CM by a CM extension of $F$ contained in $E$ (over any extension of $L$), 
    \item[(iii)] $A_0$ has CM by a CM extension of $F$ contained in $E$, but $H$ does not contain the corresponding reflex field.
\end{itemize}
Then for all $\sigma$ and all $0\neq \omega\in\Omega_\sigma$,
$$
\log_\omega(x)\neq 0.
$$
\end{thm}
We can rewrite this conclusion as follows:
Let $\sigma' = \sigma|F$.
The differential $\omega \in \Omega_\sigma$ maps to a differential
$\omega_0\in\Omega_{0,\sigma'}$ for $\Omega_0 = \Omega^1_{A_0}\otimes_{L}\ov{\BQ}$.  Let $\chi^\sigma$ be the $\ov{\BQ}^\times$-valued character defined by composition with $\sigma$,
which we also view as being $\ov{\BQ}_p$-valued via $\iota_p$.
Then
$$
\sum_{\tau\in\Gal(H/K)}\chi^\sigma(\tau) \log_{\omega_0}(\tau(y)) \neq 0.
$$

There is also a reinterpretation that has particular interest for
modular forms (see \S\ref{ss,ex} below). Let $\Gal(\ov{\BQ}/K)$ act on $A_0\otimes_{\CO_F}\CO_E$
via its usual action on $A_0$ and by multiplication by $\chi$ on the $\CO_E$-factor. This defines a $K$-twist\footnote{The character $\chi:\Gal(\ov{\BQ}/K)\rightarrow \CO_E^\times$ defines
a homomorphism into the automorphism group of $A$ and so an element
of $H^1(K,\Aut(A))$; $A_\chi$ is just the corresponding twisted form over $K$ (see also \cite{MRS}).} $A_\chi$ of the abelian variety
$A$. Then $x \in A_\chi(K)$; the conclusion is exactly the same.

\begin{proof}
If $E$ is totally real, then this is just a special case of Theorem \ref{thm:A} and hence follows from Theorem \ref{thm,mr'}. 

Suppose then that $E$ is CM. We will show that the hypothesis that $A_0$ does not admit CM by a CM extension of $F$ shows that $\theta$ does not extend to an embedding $E\times E\hookrightarrow \End^0(A)$ and so the conclusion again follows from Theorem \ref{thm,mr'}.

We argue by contradiction, working over $\ov{\BQ}$. 
Suppose $\theta$ extends to an embedding of $E\times E$ into $\End^0(A)$. Then, as noted in \S\ref{ss, tf}, there exists
an $E$-stable abelian subvariety $B\subset A$. 
As $2\dim(B) = \dim(A) = [E:\Q]$, $B$ has CM by $E$.
As $\dim(A_0) = \deg(F)$, it follows that $$A_0 \sim B_0^r$$ for some simple abelian variety $B_0$. But then $A \sim B_0^{2rd}$ where $2d = [E:F]$.
It follows that $B\sim B_0^{rd}$ and so $B_0$ has CM by a subfield $E_0$ of $E$ such that $\End^0(B_0)=E_0$. Hence $F$ is embedded into $\End^0(A_0) = \End^0(B_0^r) = M_r(E_0)$. Then $E'= E_0F$ is a commutative subalgebra of $M_r(E_0)$. As $r[E_0:\Q]\geq [E':\Q]>[F:\Q] = r[E_0:\Q]/2$, it follows that $[E':\Q] = r[E_0:\Q] = 2\dim(A_0) = 2[F:\Q]$.  From this it easily follows that $E'/F$ is a CM extension
and that $F\cap E_0$ is the totally real subfield. In particular,
the field $E'$ can be identified with a subfield of $E$ and $A_0$ has
CM by $E'$.

It remains to deal with case (iii). By Theorem \ref{thm,mr'} we can 
assume that $\theta$ extends to an embedding $E\times E\hookrightarrow \End^0(A)$. It then suffices -- again by Theorem \ref{thm,mr'} -- to show that the trace field of this extension contains the reflex field of $A_0$ (cf.~ Remark~\ref{rmk, tf}). Appealing to the preceding paragraph, we see that the trace field is the reflex field of $B$, which is just the reflex field of $B_0$.
Similarly, the reflex of $A_0$ is also the reflex field of $B_0$.
\end{proof}

It is also possible to formulate a result in the case that 
$A_0$ has CM by a CM-extension $E_0/F$ that embeds in $E$
and $H$ contains the reflex field of this action. 
Replacing $A_0$ by an isogeneous abelian variety we may assume that 
$A_0$ has an action by the ring of integers $\CO_{E_0}$.
Let
$$
A_1 = A_0\otimes_{\CO_{E_0}}\CO_E \ \ \text{and} \ \ 
A_2 = A_0^c\otimes_{\CO_{E_0}}\CO_E,
$$
where $A_0^c$ is the abelian variety such that the CM type has been replaced with its composition with complex conjugation (that is, 
$A_0^c = A_0\otimes_{\CO_{E_0},x\mapsto\ov{x}}\CO_{E_0}$).
The inclusion 
$$
\CO_{E_0}\otimes_{\CO_F}\CO_E \hookrightarrow \CO_E\times\CO_E, \ \
a\otimes b\mapsto (ab,\ov{a}b)
$$
induces an isogeny 
$$
A \sim A_1\times A_2.
$$
The abelian varieties $A_1$ and $A_2$ and this isogeny are defined
over the reflex field of $A_0$ for the action of $\CO_{E_0}$.
The respective CM-types $\Phi_1$ and $\Phi_2$ of $A_1$ and $A_2$ are complementary, that is, $\Phi_2= c\circ\Phi_1$ for $c$ the complex conjugation arising from $\iota_\infty$. It is then clear that
$\Omega_{A_i}\otimes_{E}\ov{\BQ}$ is identified with $\oplus_{\sigma\in \Phi_i}\Omega_\sigma$.

If $H$ contains the reflex field of the $E_0$-action on $A_0$, then
it could happen that $x\in A(H)$ has a non-torsion projection to 
one of $A_1$ or $A_2$ and not to the other.

\begin{thm}\label{thm, nv2} Suppose $A_0$ has CM by a CM-extension $E_0/F$ that embeds in $E$
and $H$ contains the reflex field of this action. 
If $x\in A(H)$ is non-torsion and has non-torsion projection to $A_i(H)$,
then 
$$
\log_\omega(x) \neq 0
$$
for all $\sigma \in \Phi_i$ and all $\omega\in\Omega_\sigma$.
\end{thm}
This is a more precise version of Theorem \ref{thm,mr3'}.

\begin{proof} This just follows from applying Theorem \ref{thm,mr2'}
to $A_i$.
\end{proof}

\begin{remark} It was an example of the set-up of this last theorem that 
partly provoked us to take a closer look at the questions of the introduction.  See \ref{ss,ex} for this example.
\end{remark}

\subsubsection{An Example}\label{ss,ex}
Our consideration of the questions addressed in this note was partly motivated by the following example, which comes up in \cite{BT,BCST}. 

%In the above notation suppose that $L=\mathbb{Q}$ and 
Let $K$ be an imaginary quadratic field. Let $\psi$ be a self-dual\footnote{That is, the composition $\psi^c$ of $\psi$ with the non-trivial automorphism of $K$ equals $\psi^{-1}\mathrm{N}_K$.}
Hecke character over $K$ of infinity type $(1,0)$ and $g$ the associated CM newform of weight two. Let $F_0$ be the Hecke field of $g$, which is totally real since $g$ has trivial central character. 
% $g$ be a weight two newform with trivial central character and with 
%Suppose $g$ is a weight $2$ modular form with trivial central character and complex multiplication by $K$. Let $\psi_g$ be the CM character of $K$ associated to $g$. 
Let $A_0$ be an abelian variety over $\Q$ in the associated isogeny class. The abelian varieties in this isogeny class all have CM by $E_0 = K\cdot F_0$; the CM action on $A_0$ is
defined over the reflex field $K$ of $A_0$. We may choose $A_0$ so that it has an action of the maximal order $\CO_{E_0}$ of $E_0$ (that is, the ring of integers). 
%be an associated $\mathrm{GL}_2$-type abelian variety over $\mathbb{Q}$ with complex multiplication by an order of the CM field  
%$F':=K\cdot F_g$.  Upon replacing $A_g$ by an isogenous abelian variety, we may assume that $A_g$ has CM by the integer ring $\mathcal{O}_{F'}$.
 %where $F_g$ is the Hecke field of $g$ which is totally real (since $g$ has trivial character). 
 Let $\chi$ be a finite order anticyclotomic Hecke character over $K$. 
 %Let the CM field $F'=F_g\cdot K'$. As before 
 %Put $E=F'(\Im(\chi))$, which is a CM field. 
 Let $E$ be the finite extension of $E_0$ generated by the values of $\chi$.

Put 
$$
A = A_0\otimes_{\CO_F}\CO_E, \ \ A_1 = A_0\otimes_{\CO_{E_0}}\CO_E,
\ \ \text{and} \ \ A_2 = A_0^c\otimes_{\CO_{E_0}}\CO_E,
$$
where $A_0^c$ denotes the change of the CM-action on $A_{0}$ by its composition with complex conjugation 
As noted before, the inclusion $\CO_{E_0}\otimes_{\CO_{F'}}\CO_E \hookrightarrow \CO_E\times \CO_E$
induces an isogeny 
$A \sim A_1 \times A_2$.
The abelian varieties and this isogeny are all defined over $K$.
The abelian varieties $A_1$ and $A_2$ both have CM by $E$ but their
CM-types are complementary. Let $\Phi_i$ be the CM-type of $A_i$.

As noted following Theorem \ref{thm,nv}, the abelian variety $A$ has a $K$-twist $A_\chi$ obtained from letting $G_K = \Gal(\ov{\BQ}/K)$ act on $\CO_E$ as multiplication by the Galois character associated with $\chi$ via class field theory. The decomposition $A \sim A_1\times A_2$ becomes
$A_\chi\sim A_{1,\chi}\times A_{2,\chi}$, with $A_{i,\chi}$ the similar twist. The $L$-functions of these abelian variety are
$$
L(A_{1,\chi},s) = \prod_{\sigma\in \Gal(E/K)} L(\sigma\circ(\psi\chi),s)
\ \ \text{and} \ \ 
L(A_{2,\chi},s) = \prod_{\sigma\in \Gal(E/K)} L(\sigma\circ(\psi^c\chi),s).
$$

Suppose now that $$L(\psi^c\chi,1)\neq 0.$$ Then this also holds
for each $L(\sigma\circ(\psi^c\chi),1)$. Then we expect from the Birch--Swinnerton-Dyer (BSD) conjecture that $A_{2,\chi}(K)$ is finite. 
And, indeed, this is known to follow from the non-vanishing of the $L$-values by theorems of Coates--Wiles \cite{CoWi} and Rubin \cite{Ru'}.  In particular, if
$x\in A_\chi(K)$ is any non-torsion point, then $x$ has torsion image in $A_{2,\chi}$ and so must have non-torsion image in $A_{1,\chi}$. In particular, $\log_\omega(x)=0$ for all $\sigma\in\Phi_2$ and all 
$\omega\in \Omega_\sigma$, and it follows from Theorem \ref{thm, nv2} that $\log_\omega(x)\neq 0$ for all $\sigma\in\Phi_1$ and $\omega\in \Omega_\sigma$.

If we further assume that $\mathrm{ord}_{s=1}L(\psi\chi)=1$, then the Heegner point
$x_\chi \in A_\chi(K)$ is non-torsion by the Gross--Zagier formula \cite{YZZ}, so the preceding applies to $x_\chi$. Note also
that if $H$ is a ring class field such that $\chi$ is a character
of $\Gal(H/K)$ and if $y\in A(H)$ is the Heegner point, then 
$$
x_\chi  = \sum_{\tau\in\Gal(H/K)} \tau(y)\otimes\chi(\tau) 
= \sum_{\tau\in \Gal(H/K)} \tau(y\otimes 1) \in A_\chi(K).
$$

\subsubsection{Towards the $p$-part of the BSD formula for $\GL_2$-type abelian varieties}\label{ss, p-BSD} 

Let $A = B_f$ with $B_f$ an abelian variety over $\BQ$ associated
to a modular newform $f\in S_2(\Gamma_0(N))$, as in the Introduction.
Let $F$ be its Hecke field, that is, the number field generated
by the Fourier coefficients $a_n(f)$ of $f$ (viewed as a subfield of $\ov{\BQ}$ via $\iota_\infty$). This is a totally real field, $\dim(A)=[F:\Q]=:d$, 
and there is an embedding $\theta: F\hookrightarrow \End^0_{\BQ}(A)$.
Replacing $A$ by an isogenous variety, we may assume that $F\cap\End(A) = \CO_F$, the ring of integers of $F$. 

An arithmetic consequence of Theorem~\ref{thm:A} is the following.

%\begin{cor}\label{cor,nv}
%For a $\GL_2$-type abelian variety $A/K$ as above, let \( x \in A(K) \) be a non-torsion point. For any embeddings $\sigma: F \hookrightarrow \ov{\BQ}_p$ and $\tau: K_v \hookrightarrow \overline{\mathbb{Q}}_p$, we have
%$$ \log_{A,\sigma}^{\tau}(x) \neq 0 .$$  In particular, the image of \( x \) in \( H_{\rm f}^1(K, V_{^\sigma \pi}) \) is non-zero.
%\end{cor}
%
%In particular, the corollary applies to Heegner points on a $\GL_2$-type modular abelian variety (cf.~\cite{GZ,YZZ}). Such a refined non-vanishing of the $p$-adic logarithm of Heegner points is an ingredient in a recent result towards the $p$-part of the conjectural Birch and Swinnerton-Dyer formula\footnote{The non-vanishing is asserted in \cite[Cor.~2.6.2]{S}, but not used in the proofs of main results therein.} for certain $\GL_2$-type abelian varieties over $\BQ$ 
%(cf.~\cite[Thm.~11.12]{BSTW}:  
%
\begin{thm}\label{p-BSD}
Let $p\nmid 2N$ be a prime and let $\lambda\mid p$ be a prime of $F$.
Let $T$ be the $\lambda$-adic Tate module of $A$.
Suppose:
 \begin{itemize}
 \item[(i)] Either $\lambda\nmid a_{p}(f)$ or $a_{p}(f)=0$.
\item[(ii)] If $\lambda\nmid a_{p}(f)$ then the associated mod $p$ Galois representation  
$\ov{\rho}: G_{\BQ}\ra \Aut_{\cO_{F_\lambda}}\ov{T}$ is absolutely irreducible and  ramified at a prime $q|| N$. 
If $a_{p}(f)=0$, then $N$ is square-free.
% In the non-CM case assume that \eqref{ram} holds.
% holds for the associated $\lambda$-adic Galois representation $\rho:G_{\BQ} \ra \Aut_{\cO} T$. 
\end{itemize}
If ${\rm ord}_{s=1}L(f,s)=1$, then 
the $\lambda$-part of the Birch and Swinnerton-Dyer conjecture for $A/\BQ$ holds true, i.e., $\rank_{\BZ}A(\BQ)=[F:\Q]$, $\Sha(A)[\lambda^{\infty}]$ is finite and 
$$
\bigg{|} \frac{L^{(d)}(A,1)}{d!\cdot \Omega_{A}\cdot R(A)}
\bigg{|}^{-1}_{\lambda}
=
\big{|}
\# \Sha(A)[\lambda^{\infty}] \cdot \prod_{\ell | N} c_{\ell}(A)
\big{|}^{-1}_{\lambda},
$$ 
where
\begin{itemize}
\item[--] $\Omega_A$ is the N\'eron period of $A$, 
\item[--] $R(A)$ is the regulator of the N\'eron--Tate height pairing on $A(\Q)$, 
\item[--] $c_{\ell}(A)$ is the associated Tamagawa number at a prime $\ell$.
\end{itemize}
In the CM case the same conclusion holds for any ordinary or a non-ordinary prime $p\nmid2N$. 
\end{thm}

This theorem is just \cite[Thm.~11.12]{BSTW}. The approach of {\it loc.~cit.} relies on the BDP formula \eqref{eq:BDP1} and, in turn, the non-vanishing of the $p$-adic logarithm (see also \cite{JSW}). This non-vanishing is supplied by this paper 
(cf.~Theorem~\ref{thm:A}). 

\begin{remark}\label{rmk,PR} The non-vanishing of $p$-adic logarithms in Theorem~\ref{thm:A} is also implicitly used in the proofs of Perrin-Riou's conjecture for $\GL_2$-type abelian varieties in \cite{ BKO,BSTW}.
\end{remark}

\section{The $p$-adic analytic subgroup theorem}
Theorems \ref{thm,mr'}, \ref{thm,mr2'}, and \ref{thm,mr3'} are all simple consequences of the following version of the $p$-adic
analytic subgroup theorem (cf.~\cite[Thm.~2.2]{FP}).  

Let $A$ be an abelian variety over $\ov{\BQ}$ and let $t_A$ be its tangent space.  Let $V_p = t_A\otimes_{\ov{\BQ}}\ov{\BQ}_p$ (via the fixed embedding $\iota_p:\ov{\BQ}\hookrightarrow \ov{\BQ}_p$), and let
$\log_A:A(\ov{\BQ}_p)\rightarrow V_p$ be the $p$-adic logarithm of $A$.
\begin{thm}\label{thm,p-AST}
Let $x \in A(\ov{\mathbb{Q}})$ and suppose $W \subset t_A$ is a 
$\ov{\BQ}$-subspace such that $\log_A(x) \in W_p=W\otimes_{\ov{\BQ}}\ov{\BQ}_p\subset V_p$.
Then there exists a commutative algebraic subgroup $B\subset A$
defined over $\ov{\BQ}$ such that 
\begin{enumerate}
    \item[(i)] $t_B \subset W$,
    \item[(ii)] $x \in B(\ov{\BQ})$.
\end{enumerate}
Replacing $B$ with its identity component, we may assume $B$ is an abelian subvariety of $A$ defined over $\ov{\BQ}$, provided we replace {\rm (ii)} with 
\begin{enumerate}
    \item[(ii)'] there exists $0\neq m \in \BZ$ such that $mx\in B(\ov{\BQ})$.
\end{enumerate}
\end{thm}

Suppose $\CO\subset\End(A)$ is an order in a number field $F$. 
Write $\CO = \BZ + \BZ\alpha_1 + \cdots + \BZ\alpha_r$ and 
consider 
$$
B' = B + \alpha_1B + \cdots + \alpha_r B \subset A. 
$$
This is an $\CO$-stable algebraic subgroup, and its tangent space
is contained in the subspace generated by the action of $\CO$ on $W$.
In particular, if $W$ is $F$-stable, then $t_{B'}\subset W$.
This leads to the following useful variant of Theorem \ref{thm,p-AST}:
\begin{thm}\label{thm,p-AST2}
Suppose there exists an embedding $F\hookrightarrow \End^0(A)$
of a number field $F$. Let $x \in A(\ov{\mathbb{Q}})$ and suppose $W \subset t_A$ is an $F$-stable $\ov{\BQ}$-subspace such that $\log_A(x) \in W_p=W\otimes_{\ov{\BQ}}\ov{\BQ}_p\subset V_p$.
Then there exists an abelian subvariety $B\subset A$
defined over $\ov{\BQ}$ such that 
\begin{enumerate}
    \item[(i)] $t_B \subset W$,
    \item[(ii)] there exists a compatible embedding\footnote{That is, $B$ is an $F$-stable abelian subvariety in our earlier terminology from
    \S\ref{ss, tf}.} $F\hookrightarrow \End^0(B)$,
    \item[(iii)] there exists $0\neq m\in\BZ$ such that $mx \in B(\ov{\BQ})$.
\end{enumerate}
\end{thm}

%\begin{remark}\label{rmk,p-AST}\noindent
%\begin{itemize}
%If \( A \) is a \( d \)-dimensional variety, then so is \( B \).
%\item 
%Suppose that \( \mathcal{O} \subset \text{End}_{\overline{\mathbb{Q}}}(A) \) is an order in a number field \( F \). Then \( \mathcal{O} \) acts on \( t_A \), hence \( V \) is an \( F \)-vector space and an \( \CO \otimes \overline{\BQ} \)-module. If \( W \) is also an \( F \)-subspace, then we may assume that \( B \) is \( \mathcal{O} \)-stable, i.e. \( \mathcal{O} \subset \text{End}_{\overline{\mathbb{Q}}}(B) \) consistently with \( B \subset A \).
%To see this, write \( \mathcal{O} = \mathbb{Z} + \mathbb{Z} \alpha_1 + \ldots + \mathbb{Z} \alpha_r \). Then 
%\[
%B' := B + \alpha_1 B + \ldots + \alpha_r B \subset A \quad 
%\] is $\mathcal{O}$-stable and
%$
%t_{B'} = t_B + \alpha_1 t_B + \ldots + \alpha_r t_B \subset W 
%$ since  $W$  is  $F$-stable.
%%\end{itemize}
%\end{remark} 

\section{Proofs of the main results}\label{s, proofs}
We now prove Theorems \ref{thm,mr'}, \ref{thm,mr2'}, and \ref{thm,mr3'}.
We keep to the notation introduced in Section \ref{s, mr}.

\subsection{Proof of Theorem~\ref{thm,mr'}}\label{ss, pf1}
Let $x\in A(H)$ be a non-torsion.
Suppose there exists some $\sigma\in\Sigma_F$ and some
$0\neq \omega\in\Omega_{\sigma}$ such that $$\log_\omega(x)=0.$$
Let $W = \ker\ell_\omega\subset V$. Since $\omega$ is $E$-stable
this is an $E$-stable $\ov{\BQ}$-subspace, and 
$\log_A(x)$ must lie in $W_p = W\otimes_{\ov{\BQ}}\ov{\BQ}_p\subset V_p$.
It then follows from Theorem \ref{thm,p-AST2} that 
there exists an abelian subvariety $B \subset A$ and an integer 
$m>0$ such that 
\begin{itemize}
\item $t_B \subset W$,
\item there is a compatible embedding $E\hookrightarrow\End^0(B)$,
\item $mx\in B(\ov{\BQ})$.
\end{itemize}
Since $x$ is non-torsion, $mx\neq 0$ and so $B$ must be non-zero.
As noted in the discussion of trace fields in \S\ref{ss, tf}, the existence
of such a $B$ then implies that $\theta$ must extend to an embedding 
$E\times E \hookrightarrow \End^0(A)$ (with $\theta$ being the restriction to the diagonal). So if no such extension exists, then we are done.

Suppose then that $\theta$ does extend to an embedding of $E\times E$.
As explained in \S\ref{ss, tf}, the trace field of any such embedding is the extension of $L$ generated by the traces of the action of $E$ on $t_B$.
As is also explained there, $2\dim(B) = [E:\Q]$

Let $g\in \Gal(\ov{\BQ}/H)$. Then $mx\in gB\subset A$. 
As the action of $E$ on $A$ is defined over $L$, $gB$ is also $E$-stable.
But then $B\cap gB$ is $E$-stable and non-zero (since $mx$ belongs to this intersection). The identity component of this intersection must then be a non-zero $E$-stable abelian subvariety. But this means its dimension must be at least $[E:\Q]/2$, which is the dimension of $B$. It follows that $gB = B$. Hence $B$ is defined over $H$, and there is a compatible
embedding $E\hookrightarrow \End_H^0(B)$. But then the trace of the action of any $e\in E$ on $t_B$ takes values in $H$ and so the trace field is contained in $H$, contradicting the hypothesis (ii) of the theorem.

This contradiction completes the proof of Theorem \ref{thm,mr'}.

\subsection{Proof of Theorem \ref{thm,mr2'}}
Since $E'$ is CM of degree equal to $2\dim(A)$ the $\tau\in \Sigma_{E'}$ such that $e_\tau\Omega\neq 0$ comprise a CM-type $\Phi$; this is seen by first noting that $H_1(A(\C),\Q)$ is one-dimensional $E'$-space and then 
using that $H_1(A(\C),\C) = t_{A/\C}\oplus\ov{t}_{A/\C}$.
In particular, the $\sigma\in\Sigma_E$ such that $\Omega_\sigma\neq 0$
are precisely the $\sigma = \tau|_{E}$ for $\tau\in\Phi$ and that $\tau$ 
is uniquely determined by $\sigma$.
It follows that any $0\neq\omega\in \Omega_\sigma$ is also $E'$-stable.

Returning to the proof of Theorem \ref{thm,mr'} in \S\ref{ss, pf1}, the non-zero abelian variety $B$ can be assumed to be $E'$-stable.  But this is impossible as
$2\dim(B)$ must be divisible by $[E':\BQ]$ and $2\dim(B)<2\dim(A)=2[E:\BQ]= [E:\BQ]$.

\subsection{Proof of Theorem \ref{thm,mr3'}}
Let $B_1,B_2\subset A$ be $E$-stable abelian subvarieties such that 
$m_1x\in B_1$ and $m_2x\in B_2$ for some non-zero integers $m_1,m_2$.
Then $0\neq m_2m_1 x \in B_1\cap B_2$, so the identity component $B_3$ of $B_1\cap B_2$ is a non-zero $E$-stable abelian variety, and $m_3m_2m_1x\in B_3$ for some non-zero integer $m_3$. It easily follows from this that there is a unique non-zero $E$-stable abelian subvariety $B$ of $A$ of minimal dimension containing $mx$ for some non-zero integer $m$.

Suppose $B\neq A$. As there are no non-zero proper 
$E$-stable abelian subvarieties of $B$,
the conclusion of the theorem now follows from Theorem \ref{thm,p-AST2} by the same arguments we employed before.

%\subsubsection{Proof of Theorem~\ref{theoremnonzero}} 
%Suppose that $\mathrm{log}_{\tau_0}(x)=0$ for an embedding $\tau_{0}: E \hookrightarrow \ov{\BQ}_p$. 

%Then proceeding as in the preceding subsection there exists an abelian subvariety $B\subsetneq A$ over $\ov{\mathbb{Q}}$ such that
%\begin{itemize}
%\item $mx\in B(\ov{\BQ})$ for some $m \in \BZ$.
%\item $t_B\subseteq \oplus_{\tau\not=\tau_{0}} V_{\tau}$.
%\item $B$ is $E$-stable.
%\end{itemize}
%It follows from Corollary \ref{unique} that $B$ is the smallest $E$-stable Abelian subvariety of $A$ containing $x$. 

%Since $B$ has CM by Lemma \ref{lemma1}, and note that
%$$t_B\otimes \ov{\mathbb{Q}}_p=\oplus_{\tau'\in\Phi_B} V_{\tau'}$$ where $(E,\Phi_B)$ is the CM type of $B$. In view of the second bullet point above note that $\tau_{0}\not\in \Phi_B$ and so this contradiction concludes the proof. 

\section{Complements} 
We finish with a few comments about what the $p$-adic analytic subgroup theorem can suggest about the $p$-adic closure of a set of $\BZ$-independent points $x_1,...,x_r\in A(\ov{\BQ})$ in $A(\ov{\BQ}_p)$,
or, equivalently, about the dimension of the $\ov{\BQ}_p$-space spanned
by $\log_A(x_1),\ldots,\log_A(x_r)$. Here, of course, 
$A$ is an abelian variety over $\ov{\BQ}$.  

\subsection{The set-up} As always, we fix an embedding
$\iota_p:\ov{\BQ}\hookrightarrow\ov{\BQ}_p$.
Let $d = \dim(A)$.
Let $\{\omega_{1},\cdots,\omega_{d}\}\subset \Omega^{1}_{A}$ be a $\ov{\BQ}$-basis of $\Omega^1_A$. Let $t_A$ be the tangent space of $A$ and let $\ell_{\omega_i}:t_A\rightarrow \ov{\BQ}$ be the corresponding
$\ov{\BQ}$-linear maps. Let $V_p = t_A\otimes_{\ov{\BQ}}\ov{\BQ}_p$ 
and let $\log_A:A(\ov{\BQ}_p)\rightarrow V_p$ be the $p$-adic logarithm
of $A$. The $\ell_{\omega_i}$ extend to $\ov{\BQ}_p$-linear maps on $V_p$ and we set $\log_{\omega_i} = \ell_{\omega_i}\circ\log_A$.

Let $\uline{x}= \{x_1,\cdots,x_r\}$ be a set of $\BZ$-linearly independent points $x_i\in A(\ov{\BQ})$ (so in particular, the $x_i$ are all non-torsion). We are interested in the $\ov{\BQ}_p$-rank of the matrix
$$
\CL_{\uline{x}} = (\log_{\omega_i}(x_j)) \in M_{d\times r}(\ov{\BQ}_p).
$$

For what to expect about this rank we recall the following definition and conjecture.

\begin{defn}[structural rank]\label{def,str} Let $F$ be a field of characteristic zero and $M\in M_{m\times n}(F)$. Choose a $\BQ$-basis $\{\ell_{1},\cdots,\ell_{t}\}$ for the entries of $M$ and write $M=\sum_{i=1}^{t} \ell_{i} M_{i}$ for $M_{i} \in M_{m\times n} (\BQ)$. Put 
$$
M_{x}=\sum_{i=1}^{t}x_{i}M_{i} \in M_{m\times n}(\BQ(x_{1},\cdots,x_{t}))
$$
for variables $x_{1}, \cdots, x_{t}$. The {\it structural rank} of $M$ is defined to be the rank of $M_{x}$ over the function field  $\BQ(x_{1},\cdots,x_{t})$.
\end{defn}
\noindent The structural rank conjecture for the matrix 
$\CL_{\uline{x}}$ posits:
\begin{conj}\label{conj,str} The $\ov{\BQ}_p$-rank of $\CL_{\uline{x}}$ equals its structural rank.
\end{conj}
This is a special case of a general conjecture (see~\cite{Ro} and~\cite[Conj.~4.4]{Da}), but it suffices for the purposes of this note.  Note that the 
$\ov{\BQ}_p$-rank of $\CL_{\uline{x}}$ is trivially bounded above by the structural rank. The force of this conjecture is that this upper bound is an equality.

\subsection{Towards the structural rank of $\CL_{\uline{x}}$}
Suppose now that $\End^0(A)=\BQ$ (so $A$ is generic).

\begin{prop}\label{prop, str} The structural rank of $\CL_{\uline{x}}$ is $\min\{d,r\}$.
\end{prop}

\begin{proof} 
Suppose that 
$$
\sum_{i=1}^{d}\sum_{j=1}^{r} a_{ij}\log_{A,\omega_{i}}(x_{j})=0
$$
for some $a_{ij}\in \BQ$. Without loss of generality we may and do assume that $a_{ij}\in \BZ$ for all $i,j$. Put $y_{i}=\sum_{j=1}^{r}a_{ij}x_{j} \in A(\ov{\BQ})$. Then 
\begin{equation}\label{eq,van}
\sum_{i=1}^{d}\log_{A,\omega_{i}} y_{i} = 0.
\end{equation}

The following argument is based on the $p$-adic analytic subgroup theorem for the $\ov{\BQ}$-abelian variety 
$$
B=A^d.
$$
Note that $\Omega^{1}_{B}=(\Omega_{A}^{1})^{\oplus d}$ and $t_{B}=t_{A}^{\oplus d}=(\ov{\BQ}^{d})^{\oplus d}$. Let $\lambda_i:t_B\rightarrow\ov{\BQ}$
be the map defined by projecting to the $i$th summand of $t_{A}^{\oplus d}$ and then
applying $\ell_{\omega_i}$.
Denote its extension to $V_{B,p}=t_{B}\otimes_{\ov{\BQ}} \ov{\BQ}_p$ by the same. Let $\lambda = \sum_{i=1}^d \lambda_i$. 
Let $\log_B:B(\ov{\BQ}_p)\rightarrow V_{B,p}$ be the $p$-adic logarithm of $B$.
Put $y:=(y_{1},\cdots,y_d)\in B(\ov{\BQ})$. In view of \eqref{eq,van} we have
\begin{equation}\label{eq,van'}
\log_{B}(y)\in \ker(\lambda).
\end{equation}
Since $\ker(\lambda) = \ker(\lambda|t_B)\otimes_{\ov{\BQ}}\ov{\BQ}_p$,
it follows from the $p$-adic analytic subgroup Theorem~\ref{thm,p-AST} that there exists an abelian subvariety $C \subset B$ over $\ov{\BQ}$ such that 
\begin{itemize}
\item $m\cdot y \in C(\ov{\BQ})$ for some $0\neq m\in\BZ$,
\item $t_{C} \subset \ker{\lambda|t_B}$.
\end{itemize}

Write $t_{C}=N\cdot t_{B}$ for some $N\in \End^{0}(B) = M_{d}(\BQ)$. Then 
$$
N \begin{pmatrix} t_{1}\\ \vdots \\ t_{d} \end{pmatrix} \in \ker{\lambda}
$$
for all $\uline{t}=(t_{1}, \ldots, t_{d}) \in t_{B}=t_{A}^{\oplus d}$. Writing $N=(n_{ij})$, we therefore have 
$$
\sum_{i=1}^{d} \ell_{\omega_i} (\sum_{j=1}^{d}n_{ij}t_{j}) = 0.
$$
Taking $\uline{t} = (0,\ldots, 0, t_j, 0, \ldots,0)$ we then have
$$
\sum_{i=1}^d \ell_{\omega_i}(n_{i,j}t_j)=0
$$
for all $t_j\in t_A$. But this means that $\sum_{i=1}^d n_{ij}\ell_{\omega_i}=0$.
But the $\omega_i$ are linearly independent over $\BQ$, hence so are the $\ell_{\omega_i}$. Therefore $n_{ij}=0$ for all $i$.  It follows that $N=0$.
This means that $C=0$ and hence that $y$ must be torsion. But this means that each $y_i$ must also be torsion. As the $x_j$ are linearly independent over $\BZ$ by hypothesis,
this in turn implies that $a_{ij}=0$ for all $i,j$. 

It follows that the entries of $\CL_{\uline{x}}$ are linearly independent over $\BQ$ and hence its strucutural rank is just the rank of the matrix $M = (x_{i,j})\in M_{r\times d}(\BQ(\{x_{ij}\}))$ for variables $x_{ij}$.  This, of course, is just $\min\{r,d\}.$
\end{proof}

As a consequence we deduce:
\begin{cor}\label{cor,str}
Suppose $\End^0(A) = \Q$, then the structural rank conjecture for $\CL_{\uline{x}}$ implies that the dimension over $\ov{\BQ}_p$ of 
$\sum_{i=1}^r \ov{\BQ}_p\log_A(x_i) \subset V_p$ equals $\min\{r,d\}$.
\end{cor}

\begin{remark}[Relation with a conjecture of Poonen and Prasad]
In the special case that $A$ is defined over $\BQ$ and $x_1,...,x_r \in A(\Q)$ with $r=\rank\,A(\BQ)$, it follows that the structural rank conjecture for 
 $\CL_{\uline{x}}$ implies that the closure of $A(\BQ)$ in $A(\BQ_p)$ is a $p$-adic analytic group of rank equal to $\min\{r,d\}$.  That the dimension is this was
 conjectured by Poonen \cite{Po} without the hypothesis that $\End^0(A)=\BQ$ (see also \cite[Conj.~1]{Pr}). He also asked a conjectural formula for the dimension if $A$ is defined over a general number field (cf.~\cite[Question~6.3]{Po}), which is answered by  \eqref{eq,str-var} below.

 \end{remark}

\subsubsection{Variant}
Suppose $\End^0(A) = F$ is a field. We then suppose that 
$x_1,...,x_r \in A(\ov{\BQ})\otimes_\BZ\BQ$ are $F$-linearly independent.
Consider $\CA = \sum_{i=1}^r F\cdot x_i \in A(\ov{\BQ})\otimes_\BZ\BQ$. 
By hypothesis, the dimension of $\CA$ over $\BQ$ is $r[F:\BQ]$.
We can ask what is the dimension over $\ov{\BQ}_p$ of the 
$\ov{\BQ}_p$-span $\CV_p$ of $\log_A(\CA)\subset V_p$.  The structural rank conjecture sheds some light on this.

Since $\CV_p$ is an $F\otimes_{\BQ}\ov{\BQ}_p$-module, it has a decomposition $\CV_p = \oplus_{\sigma\in \Sigma_F}\CV_{p,\sigma}$, with $\CV_{p,\sigma} = e_\sigma\CV_p$.  So we can focus
on the dimensions of the spaces $\CV_{p,\sigma}$.  
Let $d_\sigma$ be the $\ov{\BQ}$-dimension of $\Omega_\sigma$ and let $\omega_{\sigma,1},...,\omega_{\sigma,d_\sigma}$ be a $\ov{\BQ}$-basis of $\Omega_\sigma$.  Then the $\ov{\BQ}_p$-dimension of $\CV_{p,\sigma}$ is just
the rank of the matrix
$$
\CL_{\uline{x},\sigma} = (\log_{\omega_{\sigma,i}}(x_j))\in M_{d_\sigma\times r}(\ov{\BQ}_p).
$$
A straight-forward adaptation of the proof of Proposition \ref{prop, str} shows
that the entries of this matrix are $F$-linearly independent, so certainly $\BQ$-independent. The structural rank conjecture then predicts that the rank of 
$\CL_{\uline{x},\sigma}$ is $\min\{r,d_\sigma\}$. So we expect 
\begin{equation}\label{eq,str-var}
\dim_{\ov{\BQ}_p}\CV_p \stackrel{?}{=} \sum_{\sigma\in \Sigma_F} \min\{r,d_\sigma\}.
\end{equation}

Suppose that $A$ is defined over $\BQ$ and $\End_\BQ^0(A) = \End^0(A) = F$ is a field.
%(in this case, $F$ must be totally real or CM). 
Then $d = \dim(A) = d_0[F:\BQ]$ for some integer $d_0\geq 1$ and $d_\sigma = d_0$ for all $\sigma$. It follows from \eqref{eq,str-var} that we expect $\dim_{\ov{\BQ}_p}\CV_p$ to be $d_0[F:\BQ] = \dim(A)$ if $r\geq d_0$ and otherwise to equal $r[F:\BQ]$. Finally, note that in this case $\dim_{\ov{\BQ}_p}\CV_p$ is the dimension as a $p$-adic analytic group of the closure $\ov{\CA}\subset A(\BQ_p)$. So \eqref{eq,str-var} would imply Poonen's conjecture (see also \cite[Conj.~1]{Pr}). In the special case that $d_0=1$ (i.e., $\dim(A) = [F:\BQ]$) our main results show this. 
%This should be no surprise: In this special case, the structural rank is clearly $1$ and the structural rank conjecture is obviously true!  

The above discussion also yields Theorem~\ref{thm:B}.

%Let $\omega_1,...,\omega_d$ be a minimal set of $F\otimes\ov{\BQ}$-generators of
%$\Omega^1_A$.

% \subsection{The structural rank conjecture and a conjecture of Poonen}
% \subsubsection{Set-up}\label{ss,str} Let $A$ be an abelian variety over $\ov{\BQ}$ of dimension $d$. 

%   

\end{document}